\documentclass[reqno]{amsart}
\usepackage{url}
\usepackage{amssymb,amsthm,amsfonts,amstext,fancyhdr,amsmath,mathrsfs,epsfig,color}
\usepackage{amsmath}
\usepackage{mathptmx}       
\usepackage{helvet}         
\usepackage{courier}        
\usepackage{graphicx}
\usepackage{enumerate}
\usepackage{a4wide}
\usepackage[active]{srcltx}
\usepackage[english]{babel}
\usepackage[latin1]{inputenc}
\usepackage{color}

\newtheorem{theorem}{Theorem}

\newtheorem{definition}[theorem]{Definition}
\newtheorem{example}[theorem]{Example}

\newtheorem{lemma}[theorem]{Lemma}

 \newcommand{\mc}[1]{{\mathcal #1}}

\newcommand{\al}{\alpha}
\newcommand{\de}{\delta}

\newcommand{\si}{\sigma}

\newcommand{\cR}{\mc{R}}

\begin{document}

\title{A general renormalization procedure on the one-dimensional lattice and decay of correlations}
\author{ Artur O. Lopes \, } 
\date{\today}

\maketitle

\centerline{Inst. de Matem\'atica, UFRGS - Porto Alegre, Brasil}
\bigskip


\vspace*{0.5 cm}
%
%
\bigskip

\begin{abstract} We present a general form of Renormalization   operator $\mathcal{R}$ acting on potentials $V:\{0,1\}^\mathbb{N} \to \mathbb{R}$. We exhibit the analytical expression of the fixed point potential $V$ for such operator $\mathcal{R}$.  This potential can be expressed  in a
naturally way in terms of a certain  integral over the Hausdorff probability  on a Cantor type set on the interval $[0,1]$. This result generalizes a previous one by A. Baraviera, R. Leplaideur and A. Lopes  where the fixed point potential $V$ was of Hofbauer type.

For the potentials of Hofbauer type  (a well known case of phase transition) the decay is like $n^{-\gamma}$, $\gamma>0$

Among other things we present the estimation of the decay of correlation of the equilibrium probability associated to the fixed potential $V$ of our general renormalization procedure. In some cases we get polynomial decay like $n^{-\gamma}$, $\gamma>0$, and in others a decay faster than $n \,e^{ -\, \sqrt{n}}$, when $n \to \infty$.

The  potentials $g$ we consider here are elements  of the so called  family of Walters' potentials on $\{0,1\}^\mathbb{N} $  which generalizes a family of potentials considered initially by F. Hofbauer. For these potentials some explicit expressions for the eigenfunctions are known.

In a final section we also  show that given any choice $d_n \to 0$ of real  numbers varying with $n \in \mathbb{N}$ there exist a potential $g$ on the class defined by Walters which has a invariant probability with such numbers as the coefficients of  correlation (for a certain explicit observable function).

\end{abstract}

\section{Introduction} \label{I}

It is a classical problem in Statistical Mechanics the study of phase transitions (see \cite{FF}, \cite{F}, \cite{Bov} , \cite{Geo}, \cite{Wang1} and \cite{Wang2}) for general potentials. In several examples it is known that  on the point of phase transition one gets polynomial decay of correlation for the associated equilibrium probability. As far as we know there is no unified way for understanding all kinds of phase transitions.

One important technique on the analysis of such class of problems is the use of the renormalization operator (see for instance \cite{Ble}, \cite{Ishii},  \cite{BMM} and \cite{Car}). The potential which is fixed for the renormalization operator in most of the cases is quite relevant for the understanding of the problem.

Phase transitions in the setting of Thermodynamic Formalism has been considered from different points of view (see for instance  \cite{BLL}, \cite{FL}, \cite{Le1}, \cite{Le2}, \cite{Le3}, \cite{BL1} , \cite{BL}, \cite{L0}, \cite{L1}, \cite{CL}, \cite{CL2}, \cite{CR} and \cite{IT}).

In \cite{BLL} it was introduced a renormalization operator (acting on potentials) on the one-dimensional lattice $\{0,1\}^\mathbb{N}$ and the fixed point one  gets is  the well known Hofbauer potential. The equilibrium probability for such potential exhibit polynomial decay of correlation for a certain class of observables.

Here we  consider a more general procedure of renormalization of potentials defined on the lattice $\{0,1\}^\mathbb{N}$. We want to exhibit a fixed point potential for such procedure.
The fixed point potential will be on a class of functions previously considered by P. Walters in \cite{Wal} (which includes the Hofbauer potential).
We are particulary interested in the decay of correlation (for a certain natural function) associated to the equilibrium probability for such potential.

The renormalization operator defined here is globally defined.

In some cases (first type)  the fixed point will produce polynomial decay of correlation for some natural observable.
We point out that in this case in section 5 we will show the exact speed of decay of correlations and not just an upper bound.

In  others cases (second type) not of polynomial type (some of them faster than  $n \,e^{ -\, \sqrt{n}}$, when $n \to \infty$). In this case is an upper bound. This shows that fixed point for the renormalization operator (on the sense described here)  does not necessarily present polynomial decay of correlations (for some natural observable).

\medskip

In order to give an idea to the reader of  the kind of results we will get here we present an example: in a  particular case
the renormalization operator $\mathcal{R}$ acts on continuous potentials $U$, such that $U$ is constant on cylinder sets of the form
$ \overline{0^n\,1}$ and $ \overline{0^n\,1}$, for all $n \in \mathbb{N}.$
For each fixed $n$ the values are the same on both cylinder sets

We define $U_2= \mathcal{R}(U_1)$ in the following way:
suppose $x$ is of the form $x=(0^n\, 1\,z)\in\{0,1\}^\mathbb{N}$, then
$$U_2(x)\,=\,\mathcal{R}(U_1)\,(x) = U_1(0^{3\,n} 1\,z) + U_1(0^{3\,n-2} 1\,z). $$

In this case the fixed point potential $V$ for the renormalization operator  $\mathcal{R}$ is  the function
$$
  V(x)   = -\, \int_K \frac{1}{(n-t)^{\alpha}} d \nu(t),
$$
 when $x = (0^n \,1\,z)$, $n\in \mathbb{N},$ and where $\nu$ is the Hausdorff probability on the classical Cantor set (a subset of the interval $[0,1])$   which has Hausdorff dimension $\alpha=\frac{\log 2}{\log 3}.$

In several distinct classes of problems the fixed point potential for the renormalization operator can be expressed on an integral form associated to a fractal probability (see for instance  \cite{Ble}, \cite{BBM}, \cite{Ishii},  \cite{BMM} and \cite{L2}). Here we will also get a similar kind of expression.

\medskip

We will present now the main definitions giving  more details on the specific problems we want to analyze.

\medskip

We will consider a subclass of potentials $g:\{0,1\}^\mathbb{N}\to \mathbb{R}$ which were described in \cite{Wal}.
In \cite{Wal} it is considered four sequence of real numbers
$a_n,b_n,c_n,d_n$ and constants $a,b,c,d$.
It is introduced the potential $g$  which satisfies
$$ g(0^n1z)= a_n,
g(0^{\infty})=a,
g(10^n1z)=d_n,
g(10^{\infty})=d,$$
$$
g(01^n0z)=b_n,
g(01^{\infty})=b,
g(1^n0z)=c_n,
g(1^{\infty})=c,
$$
and is assumed that $a_n \to a$, $b_n \to b$, $c_n \to c$ and $d_n \to d$.

A potential of this kind is continuous and called a  Walters' potential. Depending of the choice of the sequences some potentials are of H\"older class and others are not.  Theorem 1.1 in \cite{Wal} characterize when a potential of this form  is on the Bowen's class (which contains the H\"older's - or, Lipchitz's -  potentials).

In our model $a=0=c$ and the sequences $b_n=b$ and $d_n=d$ are constant.
Moreover, we are taking $c_n=a_n$, for all  $n$.

Given $a_2,a_3,...,a_n,...<0$ define  $\eta_{j}$, by $\eta_{j}:=e^{(a_2+a_3+...+  a_{j+1})},$ $j\geq 1.$

For the sequence $1\geq\eta_n>0$, $n \in \mathbb{N}$, $n\geq 1$, $\eta_1=1$, we assume that, $ \eta_n \to 0,$
and $\sum_{n=1}^\infty  \eta_n^\beta =W(\beta)>1$, for $\beta>1$.

Particular cases which we are interested are when $\eta_n= e^{ - n^{1/2}}$,  $\eta_n= e^{ - n^{1- \frac{\log 2 }{\log 5}}}$ or $\eta_n= n^{-\gamma}$, $\gamma>2$.

Consider the two infinite collections of cylinder sets given by
\[
	L_n =  \overline{0^n\,1}
	\quad\text{and}\quad
	R_n =\overline{1^n\,0},
	\ \ \text{for all}\ n\geq 1.
\]

 Using these parameters
we can define a continuous potential $g=g_{\beta}:\Omega\to\mathbb{R}$,
which is simply denote by $g$, in the following way: for any $x\in\Omega$ and $\beta\geq 0$
\[
g(x)
=
\begin{cases}
\beta a_n, &\text{if}\ x \in L_n,\ \text{for some}\ n\geq 2;
\\[0.1cm]
\beta a_n, &\text{if}\ x \in R_n,\ \text{for some}\ n\geq 2;
\\[0.1cm]
-\log W(\beta), &\text{if}\ x \in L_1;
\\[0.1cm]
-\log W(\beta), &\text{if}\ x \in R_1;
\\
0, &\text{if}\ x \in \{1^{\infty},0^{\infty}\}.
\end{cases}
\]

This $g$ is a  Walters' potential

In Statistical Mechanics the above parameter $\beta$ is considered the inverse of temperature.

We are mainly interested in potentials $g$ which are not of H\"older type.

A particular example where phase transition occurs is presented on section 2 of \cite{CL2}.

\begin{definition}
Given $A:\Omega= \{0,1\}^\mathbb{N} \to \mathbb{R},$
the Ruelle operator $\mathcal{L}_A$  acts on functions $\psi:\Omega \to \mathbb{R} $ in the following way
$$\varphi(x)=\mathcal{L}_A(\psi)(x)= \sum_{a=0}^1 e^{A(ax)}\psi(ax).$$

By this we mean $\mathcal{L}_A(\psi)=\varphi.$
\end{definition}
From the general result described
in page 1341 in \cite{Wal} applied to our particular case, we get that
the eigenfunction $\varphi_{\beta}$ of the Ruelle operator for $A=\beta\, g$, for $\beta\leq1$, is such that, for any $n\geq 1$

\begin{equation} \label{u1}
\varphi_{\beta}(0^n 1...)
=
\alpha
\,\,
\left(
1 + \frac{1}{\eta_n^{\, \beta}}
\,
\sum_{j=1}^\infty
\frac{\eta_{n+j}^{\, \beta}\,\,}{\lambda(\beta)^{j}}
\right),
\end{equation}
\begin{equation} \label{u2}
\varphi_{\beta}(1^n 0...)
= b(\beta)\, \,\,
\left(
1 + \frac{1}{\eta_n^{\, \beta}}
\,
\sum_{j=1}^\infty
\frac{\eta_{n+j}^{\, \beta}\,\,}{\lambda(\beta)^{j}}
\right),
\end{equation}
\[
	\varphi_{\beta}(0^\infty)=\alpha
	\qquad\text{and}\qquad
	\varphi_{\beta}(1^\infty)=b(\beta),
\]

where $\lambda(\beta)$ is the associated eigenvalue, $\alpha$ and $b(\beta)$ are constants and $\eta_{j}$ was defined before.

We claim that there exists a unique positive eigenfunction for $\beta>1$. This follows from an argument similar to the one used on page 24 of \cite{PP}.

We assume that for $\beta=1$ we get $\lambda(\beta)=1=\alpha=b(\beta)$

When $\beta=1$, it is easy to see that in the case it is well defined the probability $\rho$ on the Borel sets  of $\Omega$, in such way that, for any natural number $q\geq 1$,
\begin{equation} \label{otu}
\rho \left(\overline{0^q1}\right)= \eta_q
\qquad\text{and}\qquad
\rho\left(\overline{1^q0}\right)= \eta_q,
\end{equation}
then, $\rho$ is an eigenprobability for the dual of the Ruelle operator for $g$. This claim can be obtained by generalizing the reasoning of section 5 in \cite{CL2}.

Therefore, from a general procedure (see \cite{PP}) one gets that the probability $\mu$, such that,
\begin{equation} \label{otu1}
\mu \left(\overline{0^q1}\right)= \eta_q \,\varphi (0^n 1...)= \sum_{j=0}^\infty
 \eta_{n+j}\,\,
\qquad\text{and}\qquad
\mu \left(\overline{1^q0}\right)= \eta_q \, \varphi (1^n 0...)= \sum_{j=0}^\infty
 \eta_{n+j}
\end{equation}
is the equilibrium probability  for $g$ (maximizes pressure).
\medskip

We will show in Theorem \ref{tre} that decay of correlation $d_q$,  as a function of $q\in \mathbb{N}$,  of    the equilibrium probability $\mu$ for the potential $g$ and for the indicator of the cylinder $\overline{0}$ is of order
$$ \,\sum_{s=1}^\infty \sum_{k=0}^\infty \eta_{(k+q+s)}. $$

Given a certain sequence $d_q\to 0$, there exists (under mild conditions) a choice of $\eta_n$ such that $d_q = \,\sum_{s=1}^\infty \sum_{k=0}^\infty \eta_{(k+q+s)}.$
Indeed, first for the given sequence $d_q$, $q\in \mathbb{N}$, take inductively $c_n$ such that $c_n= d_n - d_{n+1},$ $n \in \mathbb{N}$.
In this way $\sum_{j=q}^\infty\, c_j = d_q$.

Now, given such sequence $c_n$, $n\in \mathbb{N}$, take a new sequence $ \eta_r$, such that, $\eta_r= c_r - c_{r+1},$ $r\in \mathbb{N}.$ In this case $\sum_{i=p}^\infty\, \eta_i = c_p$.

It follows that for all $q$ we get $d_q =\sum_{s=1}^\infty \sum_{k=0}^\infty \eta_{(k+q+s)}.$

Part of the results presented here were a consequence  of fruitful discussions with A. Baraviera and R. Leplaideur which contributed in a non trivial way to our reasoning. Our thanks to them.

In sections 2 and 3 we present results about the fixed point potential for a certain renormalization operator. On sections 4 and 5 we estimate the decay of correlations for the class o potentials of Walters type and we show that a large class of types of decay  of correlations
are attained by such family. On the appendix we use  results of the previous sections to show that for some of these fixed point potentials the decay is not of purely polynomial type.

\section{A general renormalization - first type} \label{la}

The potentials $f$ which we are interested are on  the Walters' class
\begin{align*}
f(0^n1z)= a_n,\quad
f(0^{\infty})=a,\quad
f(10^n1z)=d_n,\quad
f(10^{\infty})=d,\quad
\\
f(01^n0z)=b_n,\quad
f(01^{\infty})=b,\quad
f(1^n0z)=c_n,\quad
f(1^{\infty})=c,\quad
\end{align*}
and is assumed that $a_n \to a$, $b_n \to b$, $c_n \to c$ and $d_n \to d$.

We will assume that $b_n=d_n=b=d$.

The renormalization operator will act on such functions $f$.

\medskip

Fix a natural number $k\geq 2$.
We will generalize the kind of renormalization described in \cite{BLL} where $k=2$.

\medskip

Now take for $n>1$
$$
   H_k(0^n 1x) = 0^{k\,n\,-(k-2)\,}1x ,
$$
$$
  H_k(1^n 0\,y) = 1^{k\, n\,-(k-2)\,} 0y ,
$$
and
$$
   H_k(0 1^m \,x) = 0\, 1^{m}\,x ,
$$
$$
  H_k(1\,0^m y) = 1\, 0^{m} y ,
$$
and, also $H_k(0^{\infty}) = 0^{\infty}$, $H_k(1^{\infty}) = 1^{\infty}$.

Now define the renormalization operator as
$$
   R_k(V)(x) =  V (H_k(x))+  V (\sigma H_k(x)) + V (\sigma^2 H_k(x))+...+ V (\sigma^{k-1} H_k(x)).
$$
 We will show that it is possible to exhibit a fixed point $V$ for the renormalization operator $R_k$.

 Note that the renormalization operator is globally defined.

 The class of functions we consider here in the  section \ref{I} is large enough to be able to get fixed points $V$. This means we will be able to find a sequence $a_n$, $n \in \mathbb{N}$, which will define such $V$

 An interesting question is to estimate the decay of correlation of the equilibrium probability for the fixed point potential $V$. One can show (see section \ref{dec}) that for the case of renormalization of first type the potential $V$ is such that we get   polynomial decay of correlation.

\bigskip

If $V$ is fixed for the renormalization operator $R_k$, then $V(1^\infty) =V(0^\infty)=0=a=c.$

Moreover,

$$V(01^\infty) = V(01^\infty) +  V(1^\infty) +  V(1^\infty)+..+ V(1^\infty)\, =\,b ,$$
and

$$a_2=V(00\,1^\infty) = V(0^{k  + 2} 1^\infty) +  V(0^{k  +1}\,1^\infty) +...+  V(000 1^\infty)=a_3+ a_4+...+ a_{ k+2}, $$

$$a_3=V(000\,1^\infty) = V(0^{2\,k  + 2} 1^\infty) +  V(0^{2\,k  +1}\,1^\infty) +...+  V(0^{\,k  +3}\,1^\infty)=a_{k+3} + a_{k+4}+...+ a_{ 2\, k+2}.$$

As we will see once we fix the value of $a_2$ and $b$ the potential $V$ will be determined.

The Walters potential $V$ which
is the fixed point for $R_k$ should satisfy for $n \geq 2$:

$$
    a_n = a_{k\,(n-1) +2 } + a_{k\,(n-1) +1 } +...+  a_{k (n-2)+3} ,
$$
$$
   a_n=c_n =  c_{k\,(n-1) +2 } + c_{k\,(n-1) +1 } +...+  c_{k (n-2)+3} ,
$$
and
$$
  d_n = b_n=b=d
$$
are constant.

Now we solve $\alpha(2)$ as the solution of

$$ a_2= -\log \frac{\, 2 + \alpha(2)}{ \, 2 + \alpha(2)-1}.$$

Now, $\alpha(n)$ is defined by induction for $n\geq 2$,
$$ \alpha_{k\,(n-1) +2 } = \alpha_{k\,(n-1) +1 } =...=  \alpha_{k (n-2)+3},$$

and
$$ \alpha_{k\,(n-1) +2 } = k\, \alpha(n)  + (k-2).$$

\medskip

Taking $a_n=d_n$, $n \geq 2$, of the form
$$ a_n= -\,\log \frac{\, n + \alpha(n)}{ \, n + \alpha(n)-1}<0$$
one can show that we get the fixed point solution of the Renormalization operator.

We point out that $a_2$ and $b$ are free.

One can show that in this case for such $V$ there exists $\gamma$ such that
$\log \eta_{n}\sim - \gamma\,\, \log n,$ $n\geq 1.$ If   $\gamma>2$,  one gets polynomial decay of correlation of the form $n^{2 - \gamma}$ similar to the
Hofbauer case. This kind of question was previously consider in \cite{L0} \cite{L1} \cite{FL} \cite{CL2}. This can be also obtained from the general proof of section \ref{dec}.

\section{A general renormalization - second type} \label{las}

We consider in this section a far more general generalization of the renormalization operator and  once more we exhibit the fixed point potential $V$.

We point out that in many examples the fixed point potential of a renormalization operator is described by an expression obtained by the integral of a kernel with respect to a certain fractal probability (see for instance  \cite{Ble}, \cite{BBM}, \cite{Ishii},  \cite{BMM} and \cite{L2}). Here we will get a similar kind of result.

  In order to define the second type renormalization operator, we first introduce the map
  $H$ (which is on Walters' class).  Let $k$ be an integer larger or equal than $2$;
  take $h(0)=0^k$ and $h(1)=1^k$;
  hence we define
  $$
    H(x_1, x_2, x_3, \ldots) = (h(x_1) h(x_2) h(x_3) \ldots).
  $$
  It is easy to see that $\si^k H = H \si$.  Hence
$$
  H(0^n 1 \ldots) = 0^{kn} 1 \ldots
    \qquad \text{say, if $x \in L_n$, then $H(x) \in L_{kn}$}
$$
and
$$
  H(1^n 0 \ldots) = 1^{kn} 0 \ldots
   \qquad \text{say, if $x \in R_n$, then $H(x) \in R_{kn}$}
$$

 We now define the renormalization operator as the map on functions defined as
 follows: fix $l$ such that $2 \leq l \leq k$ and the constants
 $0 \leq c_1 < c_2 < \ldots < c_l \leq k$. Then
$$
  \cR\,\, V (x) :=  V(\si^{c_1}H(x)) + V(\si^{c_2}H(x)) + \cdots + V(\si^{c_l} H (x)) .
$$

One can see that the Renormalization operator of last section (the case $l=k$) is a particular case of this new one .

The equation above for the potential $V$can be rewritten, analogously to the
expressions in the previous section,  as
$$
    R\, a_n = a_{kn-c_1} + a_{kn-c_2} + \cdots  + a_{kn-c_l}
$$
(and a similar expression holds to $c_n$).

We are interested on finding a fixed potential $V$ for such renormalization operator. For a potential on the class of Walters
the fixed point equation is
$$
  a_n = a_{kn-c_1} + a_{kn-c_2} + \cdots  + a_{kn-c_l}.
$$

Then, we get the following result:
\begin{lemma} Given $N \geq 1$ then
$$
   \cR^N V(x) = \sum_j  V(\si^j H^N(x) ,
$$
where $j$ is of the form
$$
   j = b_0 k^0 + b_1 k^1 + b_2 k^2 + \cdots + b_{N-1} k^{N-1},
$$
for $b_i$ chosen on the set $\{ c_1, c_2, \ldots, c_l \}$.
\end{lemma}
\begin{proof}
Indeed, the expression
it obviously valid for $N=1$ (by definition of the operator $\cR$).
Now, admiting that it holds for $N$; we need to show that this implies the
expression for $\cR^{N+1} V$. But,
$$
   \cR^{N+1} V = \cR(\cR^N V) = \cR^N \si^{c_1} H + \cR^N \si^{c_2} H + \cdots +
      \cR^N \si^{c_l}H  =
$$
$$
  \sum_j V(\si^j H^N \si^{c_1} H ) + \ldots
$$
Now, it is easy to see (also by induction) that
$$
  H \si^P H = \si^{Pk} H^2
$$
and
$$
   H^P \si H = \si^{k^P} H^2\,\, = \,\,\si^{k^P} H^{P+1}.
$$
With this two expressions we get
$$
  H^N \si^c H = \si^{c k^N} H^{N+1} .
$$
Then, we can show that $\cR^{N+1} V $ is
$$
   \sum_j V(\si^j H^N \si^{c_1} H ) + \ldots  =
$$
$$
   \sum_j V(\si^j  \si^{c_1 k^N} H^{N+1}  ) + \ldots =
   \sum_j V(  \si^{j + c_1 k^N} H^{N+1}  ) + \ldots ,
$$
where $j = b_0 k^0 + b_1 k^1 + b_2 k^2 + \cdots + b_{N-1} k^{N-1}$; hence
the expression above is indeed
$$
   \sum_J V(  \si^J H^{N+1}  ) + \ldots ,
$$
where $J =  d_0 k^0 + d_1 k^1 + d_2 k^2 + \cdots + d_{N-1} k^{N-1} + d_N k^N$,
showing that $\cR^{N+1}$ satisfies what we claim.
\end{proof}

 Now we want to show that
 we have a fixed point of $\cR$ that
 in the neighborhood of $0^{\infty}$ behaves like
 $$
     V(0^n 1 \ldots) \sim -\,\frac{1}{n^{\frac{\log{l}}{\log{k}}}} .
 $$

As we will see in a moment it will be natural to consider the (Cantor like) set
$K\,=\, K(l, k) $ which is (using the above notation) the closure of the set
$$\{ x = \sum_{n \geq 1} a_i k^{-i}  \;\;
  \text{for $a_i \in \{c_1, \ldots, c_l  \} \,$ }   \}.
$$

Note that when $k=l$ the set K will be the interval $[0,1]$.

\begin{example} Consider the following particular case.

Define $H_3 : \Omega=\{0,1\}^\mathbb{N} \to \Omega $ by:
$$H_3(0^{c_1},1^{c_2},0^{c_3},1,...))=
(0^{3c_1\,},1^{c_2},
0^{c_3},1,\ldots),$$
and
$$
H_3(1^{c_1},0^{c_2}, 1^{c_3},0,\ldots))=
(1^{3c_1},0^{c_2}, 1^{c_3},0,\ldots).
$$

In this case  the renormalization operator is
$$
   R_3(V)(x) =  V (H_3(x))+  V (\sigma\, H_3(x)) + V (\sigma^2\, H_3(x)).
$$

If $x= (0^{c_1},1,..)$ or  $x= (1^{c_1},0,..)$ denote $V(x)= \log \frac{c_1}{c_1 -1
}.$
Note that for each $c_1$ we get that
$$ \log \frac{3\,c_1}{3\, c_1 -1} +  \log \frac{3\,c_1-1}{3\, c_1 -2} + \log \frac{3\,c_1-2}{3\, c_1 -3}= \log \frac{\,c_1}{\, c_1 -1} .$$

In this case $V$ is a fixed point for the renormalization operator.

 We can choose the signal of $V$ and then, we take $V$  in the form
$$  V(1^{c_1},0,..  )\,=\,V(0^{c_1},1,..  )=-\int_0^1 \frac{1}{c_1-t} \, dt= - \log \frac{c_1}{c_1 -1
} .$$

\end{example}

\bigskip

In order to find a good guess (in the general case) for the fixed point potential consider
$$
U(x) = \frac{1}{n^{\al}},  \qquad \text{for any $n$},
$$
when, $x=(0^{n},1,.. )$, or, $x=(1^{n},0,.. )$.

Then, by iteration of $U$ we get:
$$
  \cR^N U(x) = \sum_{j} \frac{1}{(n k^N - j)^{\al}} =
  \sum_j \frac{1}{k^{N \al}} \frac{1}{(n-\frac{j}{k^N})^{\al}},
$$
for $j = b_0 k^0 + b_1 k^1 + b_2 k^2 + \cdots + b_{N-1} k^{N-1}$
(and, $b_i \in \{c_1, \ldots, c_l \}$).

From the above seems natural to try to find a fixed point $V$ of the form
$$V(x)=-\, \int_K \frac{1}{(n-t)^{\alpha}} d \nu(t),$$
when $x=(0^n\, 1...)$,
for some probability $\nu$ on the interval.

\begin{example}

Let us fix $k=3$, $l=2$ and $c_1=0, c_2=2$. In this case  the Cantor set $K$ we generate
 is indeed the classical Cantor set. We can write $K=K_0 \cup K_2$ where
$K_0 = K \cap [0, 1/3]$ and $K_2 = K \cap [2/3, 1]$.

Note that the transformation $T(x)=3\,x $ (mod $1)$ is two to one when restricted to $K.$

Moreover, $T[0,1/3]=[0,1]=T[2/3,1].$

Let $\nu$ be the measure
supported on $K$ and consider $\alpha = \log{2} / \log{3}$, that is the
Hausdorff dimension of $K$.

Note that the Radon-Nykodin derivative of $\nu$ in each injective branch of $T$ restricted to  $K_0 \subset[0,1/3]$ and $K_1 \subset [2/3,1]$
is equal to $2$.

In this case the function
$$
  V(0^n 1...) = -\, \int_K \frac{1}{(n-t)^{\alpha}} d \nu(t)
$$
is a fixed point of the renormalization operator, i.e.,
$$
  V(x) = V(H(x)) + V(\sigma^2 H(x)).
$$
\end{example}

\begin{theorem}
Let $x$ of the form: $x = (0^n 1 \ldots)$, or $x = (1^n 0 \ldots)$,  then
the operator
$$
  \cR \,V (x) :=  V(\si^{c_1}H(x)) + V(\si^{c_2}H(x)) + \cdots + V(\si^{c_l} H (x))
$$
has a fixed point $V$ of the form
$$
  V(x) =-\, \int_K \frac{1}{(n-t)^{\al}}  d\nu(t),
$$
where $\alpha= \frac{\log l}{\log k},$ and
where $\nu$ is the measure over $K$ that maximizes the entropy for the transformation  $T(x)=k\,x $ (mod $1)$ acting on $K$
.
\end{theorem}
\begin{proof} The reasoning is similar to the last example.
Consider $\alpha= \frac{\log l}{\log k}.$

The transformation  $T(x)=k\,x $ (mod $1)$ acting on $K$ is $l$ to one. The maximal entropy probability $\nu$ has entropy $\log l$.

There exists $l$ sets $K^1, K^2,...,K^l\subset [0,1]$ such that $T (K^j \cap K)=K$, and by considering $T$ restricted to $K^j=[\,c_j/k , (c_{j}+1)/k\,]$, $j=1,2,...,l$, we get an injective map. Consider the induced probability $\nu_j=T^* (\nu)$ on $K_j$.
The Radon-Nykodin derivative of $\nu_j$ with respect to $\nu$ in each set $K^j$ is equal to $ l$.

If $x = (0^n, 1,...)$ (or, $x = (1^n, 0,...)$) we want to show that
$$
 V(x)=  -\,\int_K \frac{1}{(n-t)^{\alpha}} d \nu(t) =-\, \sum_{j=1}^l \int_K \frac{1}{(\,(k\,n - c_j)-t\,)^{\alpha}} d \nu(t)=  \cR V (x).
$$

For each $j$ we have that
$$\int_K \frac{1}{(\,(k\,n - c_j)-t\,)^{\alpha}} d \nu(t) = \int_K \frac{1}{k^\alpha\,(n - \frac{c_j+t}{k}\,)^{\alpha}} d \nu(t) = \int_K \frac{1}{l\,(n - \frac{c_j}{k} - \frac{t}{k}\,)^{\alpha}} d \nu(t). $$

For each $j$ we consider the change of coordinate $t \to t/k$ on the set $K^j$.

Then, we get that
$$\int_K \frac{1}{\,(n - \frac{c_j+t}{k}\,)^{\alpha}} d \nu(t)= \int_{K^j} \frac{1}{\,(n - y\,)^{\alpha}}\, l\, d \nu(y) .$$

As
$$
   \int_K \frac{1}{(n-t)^{\alpha}} d \nu(t) = \sum_{j=1}^l \int_{K^j} \frac{1}{(\,n-t\,)^{\alpha}} d \nu(t)
$$
we get the claim.

\end{proof}

Note that
$$
  -\,\int_K \frac{1}{(n-t)^{\al}}  d\nu(t)\sim -\,n^{-\al},
$$
for large $n$. In this case $\eta_n$ will be of order $e^{-\,n^{1-\alpha}}.$  When $k=5$ and $l=2$ we will get that the equilibrium state for  fixed point $V$ has a decay faster than $ n \,e^{ -\, \sqrt{n}} $ (see Appendix).

\section{the associated Jacobian}

Now, we are interested on the estimation of the decay of correlation associated to equilibrium probabilities associated to the fixed point potentials we get via the renormalization operator. We will show that it is not always of polynomial type.
In fact, we will present the decay of correlation for a general potential of Walters type.

Suppose $g$ is of Walter type. If we add a constant to a potential its equilibrium state will not change. Then, we can assume without lost of generality that the potential $g$ has pressure zero. We want to estimate the decay for the associated equilibrium probability (which is invariant). We have to use the eigenfunction in order to get an associated normalized potential (see \cite{PP}). This normalized potential is sometimes called the Jacobian of the equilibrium probability.

The associated Jacobian $J$ for the equilibrium probability for the potential defined by $g$
at the inverse temperature $\beta=1$ is such that
$\log J= g + \log \varphi- \log (\varphi \circ \sigma)$, where $\varphi$ is the eigenfunction of the Ruelle operator for $g$ (see (\ref{u2}).

We will define $r(q)$ for $q\geq 1$.

We
split the expression of $\log J(x)$ in five cases:

\begin{itemize}

\item[a)]
for $q\geq 2$ and $x \in L_q$ or $x\in R_q$ we have
\begin{align*}
\log J (x)
&=
\beta a_q +
\log \left(
	1+ \eta_q\,\sum_{n=2}^\infty \eta_{(n+q-1)}
\right)
\\
&
\qquad\qquad\qquad\qquad\qquad\quad\quad
-\log \left(
	1+ \eta_{(q-1)}\,\sum_{n=2}^\infty \eta_{(n+q-2)}
\right)
\\
&:=
\beta a_q + \log r(q)- \log r(q-1),
\end{align*}
\item[b)]
for $x = (0 1^q\,0...) \in L_1$ we have
\[
\log J (x)
=
-\log \left(1+ \eta_q\,\sum_{n=2}^\infty \eta_{(n+q-1)}
\right)
=
-\log(r(q)),
\]
\item[c)]
for $x = (1 0^q\,1...) \in R_1$ we have
\[
\log J (x)
=
-\log \left(1+ \eta_q\,\sum_{n=2}^\infty \eta_{(n+q-1)}  \right)
=
-\log(r(q)),
\]
\item[d)] for $x=0^\infty$ or $x=1^\infty$ we have
$J(x)=1$,

\item[e)] for $x=10^\infty$ or $x=01^\infty$ we have
$ J(x)=0$.
\end{itemize}
Note that $r(1) \eta_1= W.$

It follows from \cite{Wal} that when $J$ is well defined the
equilibrium probability $\mu$ for $g$ has Jacobian $J$ and
$$ \mathcal{L}_{\log  J}^* (\mu)=\mu.$$
\bigskip

\section{decay of correlation} \label{dec}

In this section we want to estimate the decay of correlation of the observable
$I_{\overline{0}}$ for the equilibrium probability $\mu$ at
the critical inverse temperature $T$ where $\beta=1=\frac{1}{T}.$

When the potential on the Walters' class is H\"older the decay is exponential (there exists an spectral gap) according to \cite{PP}.

We assume from now on that $\beta=1$.

We will proceed in a similar fashion as in \cite{FL} and \cite{CL2}: we will get the main result via the renewal equation.

We will estimate the asymptotic behavior  with $q$ of the decay of correlation of the observable $I_{\overline{0}}$
\[
\int_{\Omega}
\, (I_{\overline{0}} \circ  \sigma^q)
\,[\,I_{\overline{0}}\, - \mu_1(\overline{0})\,] \, d\mu
\,\sim
\,\alpha(q),
\]
where $\mu$ is the equilibrium probability for $g$, when $\beta=1$.

We will show, as a particular case, that if $\eta_q\sim q^{-\gamma}, $ with $\gamma>2$, then, $\alpha(q)$ goes to zero like
$q^{2-\gamma} $. The next result is quite general.

\medskip

\begin{theorem} \label{tre} The decay of correlation as a function of $q$ for     the equilibrium probability $\mu$ (for the potential $g$) for the indicator of the cylinder $\overline{0}$ is

$$ \sum_{s=1}^\infty \sum_{k=0}^\infty \eta_{(k+q+s)} $$
in the case is well defined.

\end{theorem}

{\bf Proof:}
The technique is similar to the one employed in \cite{FL} and \cite{CL2}.

First we need to estimate the asymptotic limit
\[
\mu_1(\overline{0})-\mathcal{L}_{ \log J}^{q} (I_{\overline{0}}) (01..)
.
\]

We claim that
\begin{equation}\label{can1}
\mu_1(\overline{0})-  \mathcal{L}_{ \log J}^{q} (I_{\overline{0}}) (01..) \sim \sum_{n=1}^q\,  \sum_{j=q+ 1}^\infty \eta_{j}.
\end{equation}

We will present the proof of this fact later.

\bigskip

Now, for a fixed $q$ and any $s\geq 2$
we need to evaluate the difference
\[
\mu_1(\overline{0})-
\mathcal{L}_{ \log J}^{q} (I_{\overline{0}})(0^{ s}1..).
\]

For fixed $q$ and variable $s$, let
$
B^s_q=\mathcal{L}_{ \log J}^{q}
(I_{\overline{0}})(0^{s}1..),
$
$
A(t)= \mathcal{L}_{ \log J}^{t} (I_{\overline{0}}) (01...)
$
and finally denote $p^s_n= \frac{r(s+n)\eta_{s+n}}{r(s)\eta_s}$, $n=1,..,q-2$, $s \in \mathbb{N}$.

We denote $\alpha (q,s)=\frac{ \eta_{(s+q)} \,
r(s+q)  }{\eta_s \, r(s)}.$

One can show (see figure 3 page 1092 and expression in the bottom of page 1090  \cite{FL} where a similar case is described) that for any $s,q\geq 2$
$$
B^s_q
= A_{q-1} + e^{a_{s+1}}\, \frac{r(s+1)}{r(s)}  A_{q-2}+  e^{a_{s+1}+a_{s+2} }\, \frac{r(s+2)}{r(s)}  A_{q-3}+...+e^{a_{s+1}+a_{s+2}+...+ a_{q+s-2} }\, \frac{r(s+q-2)}{r(s)}  A_{1}+
$$

$$
e^{a_{s+1}+a_{s+2}+...+ a_{q+s-1} + a_{q+s}}\, \frac{r(s+q)}{r(s)} =
$$

$$\, A_{q-1} +  \frac{r(s+1)\,\eta_{(s+1)} }{r(s)\, \eta_s }
\, A_{q-2} +\, \frac{r(s+2)\eta_{s+2}}{r(s)\eta_s}  A_{q-3}+...+
\frac{r(s+q-2)\,\eta_{(s+q-2)}}{r(s)\,\eta_s} \, A_1+ \frac{r(s+q)\, \eta_{(s+q)}\, \,
  }{r(s)\,\eta_s \, }=
$$

\begin{equation} \label{relo}\, A_{q-1} +  \frac{r(s+1)\,\eta_{(s+1)} }{r(s)\, \eta_s }
\, A_{q-2} +\, \frac{r(s+2)\eta_{s+2}}{r(s)\eta_s}  A_{q-3}+...+
\frac{r(s+q-2)\,\eta_{(s+q-2)}}{r(s)\,\eta_s} \, A_1+ \, \alpha(n,s).
\end{equation}



This is not a kind of renewal equation.

Let us now consider for $n\geq 1$ fixed, and  $s\geq 2$
the  sequence $V_n^s = \mu_1(\overline{0}) - B^s_n$.

We also introduce, for $n\geq 1$, the sequences $ V_n=\mu_1(\overline{0}) - A_n$, and

$$\, U_n^s= - \mu(\overline{0}) (\frac{r(s+1)\,\eta_{(s+1)} }{r(s)\, \eta_s }
\, +\, \frac{r(s+2)\eta_{s+2}}{r(s)\eta_s}  +...+
\frac{r(s+q-2)\,\eta_{(s+q-2)}}{r(s)\,\eta_s} \,)- \alpha(n,s).
$$

\medskip

From the equation (\ref{relo})
we deduce that for $q$ fixed
\begin{equation} \label{eer1}  \mu_1(\overline{0})- \mathcal{L}_{ \log
J}^{q} \,I_{\overline{0}}(0^{ s}1..)\,=   \mu(\overline{0}) - B^s_q  = V_q^s =V_{q-1}\, + V_{q-2}\, p_1^s + V_{q-3}\,p_2^s +...+\, V_1\, p^s _{q-2} + U_q^s .
\end{equation}

Remember that (\ref{otu1}) claims that
\begin{equation} \label{uva3} \mu\left(\overline{0^s\,1}\right) = \eta_s\, \varphi_{1}(0^s 1...)= \eta_s\, \, r(s).
\end{equation}

For each fixed $q$ we will have to estimate later for each $s$ the value
$$\mu(\,\overline{0^s\,1}\,)\,
\big[
\mathcal{L}_{ \log
J}^{q} \,I_{\overline{0}}(0^{ s}1..)\, - \mu(\overline{0})
\big]=\,\eta_s\, \, r(s)
\big[
\mathcal{L}_{ \log
J}^{q} \,I_{\overline{0}}(0^{ s}1..)\, - \mu(\overline{0})
\big]=
$$
\begin{equation} \label{eer2}    V_{q-1}\,\eta_s\,  r(s) + V_{q-2}\, r(1+s)\eta_{1+s} + V_{q-3}\, r(2+s)\eta_{2+s} +...+\, V_1\,  r(s+q-2)\eta_{s+q-2} - r(s+q)\eta_{s+q} .
\end{equation}

\medskip


As mentioned before the Ruelle operator
$\mathcal{L}_{ \log J} $ is the dual (in the
$\mathcal{L}^2(\Omega,\mathcal{B},\mu_1)$ sense) of the Koopman operator $\mathcal{K} (\varphi) =
\varphi \circ \sigma$, then, using this duality, expressions (\ref{eer2}) and (\ref{est1})
we get that for fixed $q$ that

\begin{align*}
\int_{\Omega}
(I_{\overline{0}} \circ  \sigma^q)
\,[\,I_{\overline{0}}\, - \mu_1(\overline{0})\,] \, d \mu
&=
\int  I_{\overline{0}}  \, \mathcal{L}_{ \log J}^{q} [\,I_{\overline{0}}\, - \mu(\overline{0})\,] \, d
\mu
\\
&=
\int_{\Omega}
I_{\overline{0}}(x)
\big[ \mathcal{L}_{ \log J}^{q} \,I_{\overline{0}}(x)\, - \mu(\overline{0}) \big]
\, d \mu (x)
\\
&\hspace*{-1.5cm}=
\sum_{s=1}^\infty\,
\int_{\Omega}
I_{\overline{0}}(0^{ s}1..)
\big[
\mathcal{L}_{ \log J}^{q} \,I_{\overline{0}}(0^{ s}1..)\,
-\mu(\overline{0})
\big]
\, d \mu
\\
&=
\sum_{s=1}^\infty
\mu (\overline{0^s\,1}) \,
\big[
\mathcal{L}_{ \log
J}^{q} \,I_{\overline{0}}(0^{ s}1..)\, - \mu(\overline{0})
\big]
\\
& =- \sum_{s=1}^\infty\,\mu(\overline{0})\, (\,r(s+1)\,\eta_{(s+1)}
\, +\, r(s+2)\eta_{s+2}  +...+
r(s+q-2)\,\eta_{(s+q-2)}\,)
\\
&+
\sum_{s=1}^\infty
V_{q-1}\,\eta_s\,  r(s) +...+\, V_1\,  r(s+q-2)\eta_{s+q-2} - r(s+q)\eta_{s+q}\, \sim
\end{align*}

\begin{equation} \label{top}
 \sum_{j=1}^{q-1}
\,  V_{q-j}\,\, \sum_{k=0}^\infty \eta_{(k+s+j-1)}-
\,  \,\,\sum_{s=1}^\infty \sum_{k=0}^\infty \eta_{(k+q+s)} \sim - \sum_{s=1}^\infty \sum_{k=0}^\infty \eta_{(k+q+s)} .
\end{equation}

\medskip

Note that in the case $\eta_k =e^{-\, k^{1/2}}$ we get that $r(s) \eta_s$ is of order $\sqrt{s}\, e^{-\,\sqrt{s}}$ (see Appendix) which goes to zero when $s \to \infty$. Moreover,
$$  \sum_{s=1}^\infty \sum_{k=0}^\infty \eta_{(k+q+s)} \sim q\, e^{-\,\sqrt{q}}.$$

In the case $\log \eta_{n}\sim - \gamma\,\,\log  n,$ $n\geq 1,$ one easily get from the above that the decay of correlation is of polynomial type $n^{2-\gamma}$.

\bigskip

Now, we will prove that
$$
\mu(\overline{0})-  \mathcal{L}_{ \log J}^{q} (I_{\overline{0}}) (01..)=  V_q \sim \sum_{n=1}^q\,  \sum_{j=q+ 1}^\infty \eta_{j}.
$$

Using the scheme of smaller trees on the right side (see \cite{FL} and \cite{CL2}) we get
that the general term of $\mathcal{L}_{ \log J}^q (I_{\overline{0}})
(01...)$ for $1\leq j \leq q-1$ is given by
the following expression
$$
\mathcal{L}_{ \log J}^q (I_{\overline{0}}) (01\ldots)=
\frac{\eta_{(q-1)}}{W(\gamma)}
\,\mathcal{L}_{ \log J}^{1} (I_{\overline{0}})(01\ldots)+\ldots
$$
\begin{equation}
\label{pp1}
+\frac{\eta_2 }{W(\gamma)}\,\mathcal{L}_{
\log J}^{q-2} (I_{\overline{0}}) (01\ldots)+\frac{1}{W(\gamma)}
\,\mathcal{L}_{ \log J}^{q-1} (I_{\overline{0}}) (01\ldots)
\,+ \,\,\,
\frac{\eta_{(q+1)} r(q+1)}{\,\,W(\gamma)}
\end{equation}

 Denote $p_q= \frac{\eta_{q}}{W(\beta)}$, $q \geq 1$, and $\alpha(q)= \frac{\eta_{(q+1)} r(q+1)}{\,\,W(\beta)}$ and $K=\mu(\overline{0})= 1/2= \frac{\sum_q \alpha(q)}{\sum_q \, q \, p_q}.$

Define $V_q= \mu(\overline{0})- A(q)= \mu(\overline{0})-  \mathcal{L}_{ \log J}^{q} (I_{\overline{0}})
(01..)$, for $q\geq 1$.

For example, $V_2= p_1\,V_1 + \mu_1(\overline{0})\,(p_2 + p_3+...) -\alpha(3).$

We want to obtain the behavior of $V(q)$, when
$q\to\infty$. From the renewal equation (\ref{pp1})
we get another renewal equation: for
$q\geq 3$
\begin{equation}\label{newren}
V_q = \sum_{j=1}^{q-1} V_j p_{q-j}
+
\left[
\mu(\overline{0})\, \sum_{j=q}^\infty p_j -\alpha(q)
\right].
\end{equation}

We denote by $K_{q}$
the last term on the above equality, that is,
\begin{equation} \label{Oup}
K_{q}= \mu(\overline{0})\,\sum_{j=q}^\infty p_j- \alpha(q),
\end{equation}

$q \geq 1.$

Note that $K_1= \mu(\overline{0})- \alpha(1)=  \mu(\overline{0})- \mathcal{L}_{ \log J}^{1} (I_{\overline{0}}) (01..)=V_1.$

Note also that

$$\alpha(q) = \frac{\sum_{j=q+1}^\infty \eta_j}{W} $$

and
$$ K_{q}\,= 1/2 \, \sum_{j=q}^\infty \frac{\eta_{j}}{W(\beta)}  - \frac{ \sum_{j=q+ 1}^\infty \eta_{j}}{\,\,W(\beta)}\,\sim \,-\, \sum_{j=q+ 1}^\infty \eta_{j} \to 0,$$
when $q \to \infty$.


Now consider the following formal power series
\[
f(z) = \sum_{j=1}^\infty p_j\, z^{j-1},
\quad
V(z) = \sum_{j=1}^\infty V_j\, z^j,
\quad
\text{and}
\quad
K(z) = \sum_{j=1}^\infty K_j\, z^{j-1}.
\]
From the renewal equation (\ref{newren}) we get that
$ V(z) \, f(z) + K(z) = V(z)  .$
Therefore
$$ V(z) = \frac{K(z)}{1-f(z)}=  \frac{K(z)}{1-z}\, \frac{1-z}{1-f(z)}.$$

We assume that $p_n$, $ n \in \mathbb{N}$, is such that $f(z)$ is differentiable on $z=1$  (this is an assumption on  $\eta_n$, $ n \in \mathbb{N}$)
and the derivative is not zero.

Up to a bounded multiplicative constant we get
\[
V(z) \sim \frac{K(z)}{1-z},
\]
and from this we obtain (asymptotically)
$$
 V(z)\sim
\sum_{j=1}^\infty K_j\, z^j \,\,\,\,(1+ z +z^2 + z^3+...)=
$$
$$
 K_1\, \,+\, (K_1 + K_2)\, z + (K_1 + K_2 +K_3) z^2 +  (K_1 + K_2 +K_3 + K_4) z^3. ...
$$

In this way we get the following
recurrence relation: $V_n  \sim K_1
+ K_2+ K_3+...+ K_n.$

\medskip

This argument complete the proof of
\begin{equation}\label{can2}
\mu(\overline{0})-  \mathcal{L}_{ \log J}^{q} (I_{\overline{0}}) (01..)=  V_q \sim \sum_{n=1}^q\,  \sum_{j=q+ 1}^\infty \eta_{j}.
\end{equation}

This goes to zero because $\sum_q\, q \,\eta_q \, < \infty$.

When $\eta_k =e^{-\, k^{1/2}},$ we get that

\begin{equation}\label{est1}
V_k \sim k^2\, e^{-k^{1/2}} .
\end{equation}

In the same way
\begin{equation}\label{can2}
\mu(\overline{0})-  \mathcal{L}_{ \log J}^{q} (I_{\overline{0}}) (10..) \sim \sum_{n=1}^q\,  \sum_{j=q+ 1}^\infty \eta_{j}.
\end{equation}

\qed
\section{Appendix} \label{ap}

First we consider  the case $\eta_k =e^{-\, k^{1/2}}$ and we will take the estimation  (\ref{top}).

The estimate  of $\sum_{k=0}^\infty e^{- \sqrt{n+ j+k}}$, for $n,j$ fixed, can be done via the integral
$$ \int_1^\infty  e^{- \sqrt{n+ j+x}}\, dx.$$

Taking the change of coordinates $y= \sqrt{n+ j+x}$, we get the equivalent expression
$$2 \int_{\sqrt{n+j+1}}^\infty e^{-y}\, y \,dy.$$

As
$$ (e^{-y}\,\, y)\,|_{\sqrt{n+j+1} }^\infty= \int_{\sqrt{n+j+1}} \frac{d\, (e^{-y}\,\, y)}{dy}\, dy =\int e^{-y}\, dy - \int y\, e^{-y}\, dy,$$

we get that $ \int_1^\infty  e^{- \sqrt{n+ j+x}}\, dx$ is of order $\, \sqrt{n+j+1} \,e^{-\,\sqrt{n+j+1}   }.$

Now, for  $n$ fixed, we want to estimate the expression $\sum_{j=1}^\infty \,\sum_{k=0}^\infty e^{- \sqrt{n+ j+k}}= \sum_{j=1}^\infty \,\sum_{k=0}^\infty \eta_{n+ j+k} $.

In order to to that we consider the integral
$$\int_1^\infty\,  \sqrt{n+ x} \,\,e^{ -\, \sqrt{n+ x}}\, dx.$$

Considering the change of coordinates
$y= \sqrt{n+ x}$ we get that
$$ (e^{-y}\,\, y^2)\,|_{\sqrt{n+1} }^\infty= \int_{\sqrt{n+1}} \frac{d\, (e^{-y}\,\, y^2)}{dy}\, dy =2\,\int_{\sqrt{n+1}} e^{-y}\,y\, dy - \int y^2\, e^{-y}\, dy,$$

and finally
$$\sum_{j=1}^\infty \,\sum_{k=0}^\infty \eta_{n+ j+k}\,\sim\,\int_1^\infty\,  \sqrt{n+ x} \,\,e^{ -\, \sqrt{n+ x}}\, dx  \sim n \,e^{ -\, \sqrt{n}} .$$

In a similar way, if
$\eta_k =e^{-\, k^{1-\frac{\log 2}{\log 5}}}$ we get  a similar result, that is, the decay is faster than polynomial.

Indeed, $\frac{\log 2}{\log 5}= 0.4306 $, and then $1-\frac{\log 2}{\log 5}>0.5$. From this follows that
$e^{-\, k^{1-\frac{\log 2}{\log 5}}}\leq e^{-\, k^{1/2}} $. Therefore, for the case of the fixed point $V$ for the renormalization operator when $k=5$ and $l=2$, we get that the associated equilibrium state has a decay faster than
$ n \,e^{ -\, \sqrt{n}} .$ To get the exact speed is not so easy on this case.





\end{document}